\documentclass[10pt]{amsart}
\usepackage[utf8]{inputenc}
\usepackage{geometry}
\usepackage{fdsymbol}
\usepackage{natbib}

\numberwithin{equation}{section}

\theoremstyle{definition}
\newtheorem{theorem}{Theorem}[section]
\newtheorem{proposition}[theorem]{Proposition}
\newtheorem{lemma}[theorem]{Lemma}
\newtheorem{corollary}[theorem]{Corollary}
\newtheorem{definition}[theorem]{Definition}
\newtheorem{remark}[theorem]{Remark}

\theoremstyle{plain}
\numberwithin{theorem}{section}

\def\CC{{\mathbb{C}}}
\def\PP{{\mathbb{P}}}
\def\QQ{{\mathbb{Q}}}
\def\ZZ{{\mathbb{Z}}}

\DeclareMathOperator{\Bl}{Bl}
\DeclareMathOperator{\Br}{Br}

\newcommand{\Q}{\mathbb{Q}}
\newcommand{\Z}{\mathbb{Z}}
\newcommand{\C}{\mathbb{C}}
\newcommand{\NS}{\operatorname{NS}}
\newcommand{\Rat}{\mathrm{Rat}}

\newcommand{\id}{\mathrm{id}}
\newcommand{\Gr}{\operatorname{Gr}}
\newcommand{\End}{\operatorname{End}}
\newcommand{\Ccal}{\mathcal{C}}
\newcommand{\rk}{\operatorname{rk}}
\newcommand{\xspecial}{X\in\mathcal{C}_d}

\title{Hasse Obstructions to Rationality for special fourfolds}
\author{Aideen Fay}
\date{\today}

\begin{document}

\begin{abstract}
For every nonempty Hassett divisor $\mathcal C_d$ parameterizing special cubic fourfolds, we consider the quaternion class $\beta_d=(d/2,-3)$ in the two-torsion Brauer group. We prove that if a Hodge-general member of ${C}_d$ is rational then $\beta_d=0$ -- or, equivalently, that Huybrechts' twisted-K3 condition $(**')$ holds. Consequently, a very general member of $\mathcal C_d$ is irrational if $\beta_d\ne0$. Thus, a very general cubic fourfold containing a smooth cubic scroll or a Veronese surface is irrational, and hence so is a very general K\"uchle fourfold of type $(\mathrm{c7})$. We also obtain an analogous obstruction for Hodge--special Gushel--Mukai fourfolds: a very general member in the discriminant--$d$ locus can be rational only if $d$ is a sum of two squares.  Consequently, a very general Gushel--Mukai fourfold containing a cubic scroll is irrational.
\end{abstract}
\maketitle

\section{Introduction}
Let \(X\subset \PP^5\) be a smooth complex cubic fourfold with hyperplane class $h$. Define
\[
A(X)_{\Z}:=H^4(X,\Z)\cap H^{2,2}(X),\qquad
A(X) := A(X)_\ZZ \otimes_{\Z} \Q, \qquad
T_{X,\Z}
  :=A(X)_{\Z}^{\perp}\subset H^4(X,\Z),
\]
taking orthogonal complements with respect to the intersection pairing. Then $T_X:=T_{X,\Z}\otimes_{\Z}\mathbb Q$ is the rational transcendental Hodge structure of $X$. Let $\End_{\mathrm{Hdg}}(T_{X})$ denote the algebra of $\Q$-linear Hodge endomorphisms of $T_{X}$.

Hassett's divisor $\mathcal C_d$ parameterizes cubic fourfolds where $A(X)_{\Z}$ contains a primitive rank-two sublattice $K_d$ containing $h^2$ and having discriminant $d$ \cite[\S3.2]{hassett}. Such a cubic $\xspecial$ is called \emph{special of discriminant d}, and the locus of such cubics is nonempty if and only if $d$ satisfies
\begin{equation}
d>6
\qquad\text{and}\qquad
d\equiv 0\ \text{or}\ 2\pmod 6.
\tag{*}\label{eq}
\end{equation}
We define
\[
\mathcal{C}_d^{\mathrm{Hdg}} := \left\{ X\in\mathcal{C}_d: \operatorname{rk} A(X)_{\Z}=2,\ \operatorname{End}_{\mathrm{Hdg}}(T_{X}) =\mathbb{Q}\cdot\mathrm{id}_{T_X} \right\},
\]
and call its members \emph{Hodge-general}. As shown in Lemma~\ref{lem:generic}, a very general member of every nonempty $\mathcal C_d$ is Hodge-general.

For every such divisor, consider the quaternion class
\[
\beta_d:=\left(\frac d2,-3\right)\in\Br(\Q)[2],
\]
where \((a,b)\) denotes the Brauer class of the quaternion algebra
\((a,b)_{\Q}\).  The vanishing of $\beta_d$ is equivalent to Huybrechts' twisted-K3 condition $(**')$ recalled below.
\begin{theorem}\label{thm:main}
If a Hodge-general $X\in\mathcal C_d$ is rational, then $\beta_d=0$. Equivalently, $d$ satisfies $(**')$. Consequently if $\beta_d\ne0$, a very general member of $\mathcal C_d$ is irrational. In particular, when $\beta_d\ne0$, the rational locus \(\Rat\subset\mathcal C\) satisfies

\begin{multline}\label{eq:containment}
\Rat\cap\mathcal C_d\subseteq\{X\in\mathcal C_d:\rk A(X)_{\Z}\ge3\} \,\cup \\ \{X\in\mathcal C_d: \text{$\rk  A(X)_{\Z}=2$ and $T_X$ has non-trivial real multiplication}\},
\end{multline}
where $\mathcal{C}$ is the moduli space of smooth complex cubic fourfolds.
\end{theorem} 

By the norm criterion for $\Q(\sqrt{-3})/\Q$, this means that the class $\beta_d$ vanishes if and only if $v_p(d/2)$ is even for every prime $p \equiv 2 \pmod 3$. Indeed, this is equivalent to the condition introduced by Huybrechts in~\cite[Theorem~1.3]{Huybrechts2015TheKC}
\begin{equation}
\frac d2=\prod_i p_i^{n_i},
\qquad
n_i\equiv0\pmod 2
\quad\text{whenever}\quad
p_i\equiv2\pmod 3.
\tag{$**'$}\label{eq:starprime}
\end{equation}
Huybrechts proves that $X\in\mathcal C_d$ for some $d$ satisfying
$(**')$ if and only if there exist a twisted K3 surface $S$ with Brauer class $\alpha\in\Br(S)$, and an integral Hodge isometry
\begin{equation}\label{eq:hub-mukai}
\widetilde{H}(\mathcal{A}_X, \Z)\simeq\widetilde{H}(S, \alpha, \Z),
\end{equation}
where $\mathcal{A}_X$ is the Kuznetsov component of $X$. 

As shown in \cite[Equation~1.1]{Huybrechts2015TheKC}, this isomorphism of integral Mukai lattices  induces, up to sign, an integral Hodge isometry of transcendental lattices 
\begin{equation}\label{eq:hub}
T_{X,\mathbb Z}(1)\simeq T(S,\alpha,\Z)
 \simeq \operatorname{Ker}(\alpha\,:T(S,\Z)\longrightarrow {\Q}/{\Z}).
\end{equation}
Huybrechts also proves that $(**')$ is equivalent to the Fano variety of lines $F$ on $X$ being
birational to a moduli space of stable \(\alpha\)-twisted sheaves on $S$; see \cite[Proposition~4.1]{Huybrechts2015TheKC}.

\begin{corollary}
\label{cor:rational-addington}
Let $X\in\mathcal C_d$ satisfy $\rk A(X)_{\Z}=2$. Then the following conditions are equivalent:
\begin{enumerate}
    \item \(d\) satisfies \((**')\);
    \item there exist a projective K3 surface $S$ and an isometry $T_X(1)\simeq T_S$ of polarized rational Hodge structures;
    \item There exist a projective K3 surface $S'$ and $\alpha' \in Br(S')$ such that $F$ is birational to a moduli space of stable \(\alpha'\)-twisted sheaves on \(S'\).
\end{enumerate}
Moreover, if $X\in\mathcal C_d^{\mathrm{Hdg}}$ is rational, then
these conditions hold.
\end{corollary}
The equivalences follow from the results of Huybrechts and the Hasse invariant comparison established in Section~\ref{sec:hasse}. In a sense, this is the rational-coefficient, twisted counterpart to the result proved by Addington in~\cite[Theorem~1]{Addington_2016}, which shows that Hassett's stronger condition 
\begin{equation}
4 \nmid d, \quad 9 \nmid d, \quad \text{and} \quad p \nmid d \text{ for every odd prime } p \equiv 2 \pmod 3,
\tag{$**$}\label{eq:hassett}
\end{equation}
is equivalent to the existence of an untwisted K3 surface $S$ and an integral Hodge isometry
$T_{X,\Z}(1)
\simeq
T(S, \Z)$,
and also to $F$ being birational to a moduli space of stable sheaves on $S$.

We apply Theorem~\ref{thm:main} to prove irrationality for the very
general members of infinitely many Hassett divisors. The smallest $d$ for which our obstruction is nonzero is $12$. Consequently, we get the following new results: 
\begin{corollary}\label{cor:c12}
A very general member of $\mathcal C_{12}$ is irrational.
Equivalently, a very general cubic fourfold containing a cubic scroll
is irrational.
\end{corollary}
\begin{corollary}\label{cor:c20}
A very general cubic fourfold in $\Ccal_{20}$ is irrational. Hence a very general K\"uchle fourfold of type $(\mathrm{c7})$ is irrational.
\end{corollary}


We also obtain an analogous obstruction for \emph{Hodge--special} Gushel--Mukai fourfolds; see Theorem~\ref{thm:gm} below. To our knowledge, this gives the first irrationality results for
very general members of fixed Noether--Lefschetz divisors of
Gushel--Mukai fourfolds. In particular, with $d=12$, we obtain:

\begin{corollary}\label{cor:scroll}
A very general Gushel--Mukai fourfold containing a cubic scroll is
irrational.
\end{corollary}

Debarre--Iliev--Manivel studied Hodge--special Gushel--Mukai fourfolds in~\cite[\S6]{debarre2015special}. For these fourfolds, Pertusi proved that the condition $\gamma_d=0$ in Theorem~\ref{thm:gm} is equivalent to the existence of an associated twisted K3 surface; see \cite[Theorem~1.1]{pertusi2019double}. This is compatible with the Kuznetsov--Perry conjecture, which states that the GM category of every rational Gushel--Mukai fourfold is
equivalent to the derived category of a K3 surface
\cite[Conjecture~3.12]{Kuznetsov2016DerivedCO}.  Theorem~\ref{thm:gm} therefore provides an obstruction to rationality without assuming this conjecture.

Finally, we apply our obstruction to  Verra fourfolds in Corollary~\ref{cor:verra} to show that a Hodge-general cubic fourfold in $\mathcal C_d$, or a Hodge-general Gushel--Mukai fourfold in $\mathcal D_d$, can be
birational to a smooth Verra fourfold only if $d$ satisfies the respective condition $(**')$ of Huybrechts or Pertusi for a twisted K3 realization. This implies:

\begin{corollary}\label{cor:verra-examples}
None of the following are birational to a smooth Verra fourfold:
\begin{itemize}
    \item a very general cubic fourfold containing a smooth cubic scroll;
    \item a very general cubic fourfold containing a Veronese surface;
    \item a very general Gushel--Mukai fourfold containing a cubic scroll;
    \item a very general K\"uchle fourfold of type $(\mathrm{c7})$.
\end{itemize}
\end{corollary}
The entry points for our work are the results of Gu\'er\'e~\cite[Theorem~56]{guere} and Benedetti--Manivel--Perrin~\cite[Theorem~13]{benedetti2026quantumcohomologyirrationalitygushelmukai} for cubic and Gushel--Mukai fourfolds, respectively. These in turn build upon the recent \emph{Hodge atom} birational invariants introduced by Katzarkov--Kontsevich--Pantev--Yu in \citep{KKPY}, and the decomposition theorem for quantum cohomology under blowups by Iritani~\citep{iritani2023quantum, iritani2026notesdecompositiontheoremblowups}.

\section{K3 realizations and rational Hodge similitudes}
\label{sec:iso}
Throughout, we work with smooth complex projective varieties over $\C$. Recall that a \emph{polarization} of a weight-$n$ rational Hodge structure $V$ is a morphism $\psi\colon V \otimes V \to \Q(-n)$ whose Weil form is symmetric and positive-definite. A \emph{Hodge isometry} is an isomorphism of Hodge structures compatible with the polarizations. Orthogonal complements are taken over $\Q$ with respect to the cup-product forms $q_X$ and $q_S$. 

Recall that a smooth complex cubic fourfold $X$ has Hodge numbers $h^{3,1}(X)=h^{1,3}(X)=1$ and $h^{2,2}(X)=21$. Its
primitive cohomology is defined as
\[
H^4(X,\Q)_{\text{prim}}=(\Q h^2)^\perp\subset H^4(X,\Q),
\]
with signature $(20,2)$. Define the primitive algebraic cohomology $N_X$ as 
\[
N_X:=A(X)\cap H^4(X,\Q)_{\text{prim}}.
\] 
The Hodge--Riemann relations give the orthogonal decompositions 
\begin{equation}\label{eq:decompositions}
A(X) =\Q h^2\perp N_X,
\qquad
H^4(X,\Q)_{\text{prim}} =N_X\perp T_X.
\end{equation}
In the Hodge-general case, $T_X$ therefore has rank-$21$ and signature $(19,2)$, so $T_X(1)$ is a Hodge structure of K3 type polarized by $-q_X$.  

For a K3 surface $S$, let 
\[
T_{S,\Z}:=\operatorname{NS}(S)^\perp\subset H^2(S,\Z),
\qquad
T_S:=T_{S,\Z}\otimes_{\Z}\Q.\]
Equivalently, $T_S$ is the smallest rational sub-Hodge structure of $H^2(S,\Q)$ containing $H^{2,0}(S)$; see \cite[Definition~3.2.5]{lecturesHuybrechts}.

\begin{definition}\label{def:hg}
The Hodge-general locus in $\mathcal C_d$ is
\[
\mathcal{C}_d^{\mathrm{Hdg}} := \left\{ X\in\mathcal{C}_d: \operatorname{rk} A(X)_{\Z}=2,\ \operatorname{End}_{\mathrm{Hdg}}(T_{X}) =\mathbb{Q}\cdot\mathrm{id}_{T_X} \right\},
\]
Its members are called Hodge-general.
Let $\End_{\mathrm{Hdg}}(T_{X})$ denote the algebra of
$\Q$-linear Hodge endomorphisms of $T_{X}$.
\end{definition}
 
\begin{lemma}\label{lem:rm}
Let \(X\in\mathcal C_d\) satisfy $\rk A(X)_{\Z}=2$.  Then
\[
\End_{\mathrm{Hdg}}(T_X)\neq \mathbb Q\cdot\id_{T_X}
\quad\Longleftrightarrow\quad
T_X \text{ has nontrivial real multiplication}.
\]
\end{lemma}

\begin{proof}
Simplicity of \(T_X\) follows
from \(h^{3,1}(T_X)=1\) and the absence of nonzero rational classes of
type \((2,2)\) in \(T_X\). Since \(T_X\) is simple and of K3 type,
Zarhin's theorem~\cite[Theorems~1.5.1 and~1.6(a)]{Zarhin1983} shows that $E:=\operatorname{End}_{\mathrm{Hdg}}(T_X)$ is either totally
real or CM. The action of $E$ on $T_X$ gives $\dim_{\Q}T_X=[E:\Q]\dim_E T_X$.  Thus $[E:\Q]$ is odd. A CM field has even degree over $\Q$, so $E$ must be totally real. 
\end{proof}
 


\begin{proposition}\label{prop:source}
Let $X$ be a rational cubic fourfold. Then there exist a projective K3 surface $S$ and an isomorphism of rational Hodge structures
$$
f\,\colon T_X(1)\xrightarrow{\sim}T_S.
$$
If, moreover, 
$$\operatorname{End}_{\mathrm{Hdg}}(T_{X}) =\mathbb{Q}\cdot\mathrm{id}_{T_X},$$
then there exists $c\in\Q_{>0}$ such that
\begin{equation}\label{eq:similitude}
q_S(fx,fy)=-c\,q_X(x,y)\qquad \forall\,x,y\in T_X(1).
\end{equation}
\end{proposition}

\begin{proof}
By Gu\'er\'e's theorem \cite[Theorem 56]{guere}, rationality gives a projective K3 surface $S$ and a Hodge isomorphism 
\begin{equation*}
\varphi\,\colon H^4(X,\Q)_{\text{prim}}\xrightarrow{\sim}H^2(S,\Q)(-1).
\end{equation*}
Taking a Tate twist, $\varphi(1)$ identifies the $(2,0)$-lines on both sides. Since the transcendental part of a Hodge structure of K3 type is the smallest rational sub-Hodge structure containing this line, $\varphi(1)$ restricts to an isomorphism
$$
f\colon T_{X,\Q}(1)\xrightarrow{\sim}T(S)_{\Q}.
$$

The forms $-q_X$ and $f^*q_S$ are both polarizations of $T_X(1)$. Nondegeneracy gives a unique $a\in\End_{\mathrm{Hdg}}(T_{X})$ such that
$$
q_S(fx,fy)=-q_X(ax,y).
$$
Since both forms are Hodge tensors, $a$ is a Hodge endomorphism.
Thus
\[
a\in\operatorname{End}_{\mathrm{Hdg}}(T_X)
=\Q\cdot\mathrm{id}_{T_X},
\]
so $a=c\,\mathrm{id}_{T_X}$ for some $c\in\Q$. Since both forms are
polarizations, $c>0$. This proves \eqref{eq:similitude}.
\end{proof}

\begin{lemma}\label{lem:generic}
The very general member of every nonempty Hassett divisor $\mathcal{C}_d$ lies in $\mathcal{C}_d^{\mathrm{Hdg}}$.
\end{lemma}
\begin{proof} 

Let $\mathcal C_d^\circ\subset\mathcal C_d$ be a smooth dense open
subset. Let
$p\,\colon S\longrightarrow\mathcal C_d^\circ$ 
be a connected finite \'etale cover parameterizing labelled cubic fourfolds, and let $\pi\,\colon\mathcal X\to S$ be the universal family. The labelling induces a rank-two sub-local system
$\mathbb K\subset R^4\pi_*\ZZ$. Its orthogonal complement $\mathbb T_{\Q}:=(\mathbb K\otimes\Q)^\perp$ is a polarized variation of rational Hodge structures.

Fix $x_0 \in S$ and a simply connected analytic neighbourhood $U\subset S$ of $x_0$. The underlying local system $\mathbb T_{\Q}$ is constant over $U$, with fibre $T_{\Q}$.

Define the period domain
\[
\Omega_T=
\left\{
[\omega]\in\PP(T_{\C}) :
q(\omega,\omega)=0,\ 
q(\omega,\overline{\omega})<0
\right\},
\] 
and the period map
\[
\mathcal P\,\colon U\longrightarrow\Omega_T,
\qquad
x\longmapsto[\omega_x].
\]
For special cubic fourfolds, the period map is an open immersion
\cite[Theorem~2.2.1]{hassett}. 

Set
$$
U^{\mathrm{Hdg}}
:=
U\cap p^{-1}\bigl(\mathcal C_d^{\mathrm{Hdg}}\bigr).
$$
If $x\in U$ is not Hodge-general, then either a nonzero
$v\in T_{\Q}$ becomes of type $(2,2)$, or a nonscalar
$M\in\End_{\Q}(T_{\Q})$ becomes a Hodge endomorphism. Accordingly,
define
$$
\mathcal H_v
:=
\{x\in U:q(v,\omega_x)=0\}
$$
and
$$
\mathcal E_M
:=
\{x\in U:M\in\End_{\mathrm{Hdg}}(\mathbb T_{\Q,x})\}.
$$
Then
$$
U\setminus U^{\mathrm{Hdg}}
\subseteq
\bigcup_{v\in T_{\Q}\setminus\{0\}}\mathcal H_v
\cup
\bigcup_{M\in\End_{\Q}(T_{\Q})\setminus\Q\cdot\id}
\mathcal E_M.
$$

Let us show that these exceptional loci are contained in a strict closed analytic subset of $U$. For $v\neq0$, the hyperplane $q(v,\omega)=0$ defines a strict closed analytic subset of $\Omega_T$. Since the period map is an open immersion, $\mathcal H_v$ is a strict closed analytic subset of $U$.

Now suppose that $M$ is nonscalar. If $x\in\mathcal E_M$, then $M$ preserves the period line, so
$$
[\omega_x]\in
\bigcup_{\lambda\in\operatorname{Spec}(M)}
\PP\bigl(\ker(M-\lambda\id)\bigr).
$$
If this union contained an open subset of $\Omega_T$, its Zariski closure would contain the quadric $q=0$. By irreducibility, this quadric would lie in a single eigenspace of $M$. Since it spans $\PP(T_{\C})$, $M$ would then be scalar, which contradicts our initial assumption. Hence $\mathcal E_M$ is contained in a strict closed analytic subset of $U$.

Finally, $T_{\Q}$ and $\End_{\Q}(T_{\Q})$ are finite-dimensional over
$\Q$, so there are only countably many choices for $v$
and $M$. Hence $U\setminus U^{\mathrm{Hdg}}$ is contained in a countable union of strict closed analytic subsets of $U$. Such a union cannot cover a complex manifold, so $U^{\mathrm{Hdg}}\neq\varnothing$.

By Cattani--Deligne--Kaplan~\cite[Corollary~1.3]{CDK}, the irreducible components of the loci defined by nonzero flat sections in $\mathbb T_{\Q}(2)$ and nonscalar flat tensors in
$\mathbb T_{\mathbb Q}^{\vee}\otimes\mathbb T_{\mathbb Q}$ are
algebraic. The previous argument shows that they are strict. Since $p$ is finite, the non-Hodge-general locus is contained in the union of the closures of their images and the boundary $\mathcal C_d\setminus\mathcal C_d^\circ$. This is a countable union of strict closed subvarieties of $\mathcal C_d$. Hence a very general member of $\mathcal C_d$ lies in $\mathcal C_d^{\mathrm{Hdg}}$.

\end{proof}

\section{The Hasse--Witt Obstruction}\label{sec:hasse}

A quadratic form over $\Q$ is the pair $(V,\,q)$, where V is a finite--dimensional $\Q$--vector space, and $q: V \times V\rightarrow \Q$ is a nondegenerate symmetric bilinear form. After choosing an orthogonal basis, we may write \[
q=\langle a_1,\ldots,a_n\rangle\qquad a_i\in\QQ^\times.
\] 
The determinant and Hasse invariant of $q$ are defined respectively as
\[
\delta(q):=\prod_{i=1}^n a_i
\in\QQ^\times/\QQ^{\times 2},
\qquad
w(q):=\sum_{i<j}(a_i,a_j)
\in\Br(\QQ)[2].
\] 
Here $(a,b)\in\Br(\QQ)[2]$ denotes the Brauer class of the quaternion algebra $$H(a,b)=\Q\langle i,j\rangle/(i^2-a,j^2-b,ij+ji)$$ where $\Br(\QQ)$ is the group of Brauer--equivalence classes of
central simple $\QQ$--algebras. We use additive notation, which means that when  $(a,b)=0$, the algebra is split; equivalently $H(a,b)\simeq M_2(\QQ)$.

For each place $v$ of $\QQ$, let $(a,b)_v\in\{\pm1\}$ denote the local Hilbert symbol, so that
\[
(a,b)_v=1
\quad\Longleftrightarrow\quad
H(a,b)\otimes_{\QQ}\QQ_v\simeq M_2(\QQ_v).
\]
In particular, a single place with $(a,b)_v=-1$ proves that $(a,b)\neq0$ in $\Br(\QQ)[2]$.

We will make repeated use of the following identities throughout the remainder of the section. Let $r$ be another nondegenerate quadratic form over $\Q$, let $m=\dim_{\Q}q$, and let $\lambda\in\QQ^\times$, then
\begin{align}
w(q\perp r)
&=w(q)+w(r)+\bigl(\delta(q),\delta(r)\bigr),
\label{eq:sum}\\
w(\lambda q)
&=w(q)+\binom{m}{2}(\lambda,\lambda)
+(m-1)\bigl(\lambda,\delta(q)\bigr).
\label{eq:scale}\end{align} 
Furthermore, two quadratic forms over $\Q$ are isomorphic if and only if they have the same dimension, signature, determinant, and Hasse invariant. See
\cite[Chapter~III]{Se} for this and more details on quaternion algebras and local Hilbert symbols.  Throughout the remainder of the paper we abbreviate
\[
w(T_X):=w(q_X\vert_{T_X}),
\qquad
w(-T_S):=w((-q_S)\vert_{T_S}).
\]

\begin{remark}\label{rem:rank}
For $m=21$ both $\binom{21}{2}=210$ and $m-1=20$ are even, so $w(\lambda q)=w(q)$ for every $\lambda\in\Q^{\times}$. A similitude with multiplier $\lambda$ is an isometry onto $\lambda q$, so in rank-$21$ the Hasse class is similitude-invariant. For cubic fourfolds, this holds when $m=\operatorname{rk} T_X \equiv 1 \pmod 4$, or equivalently, $r\equiv 2 \pmod 4$, where $r=\operatorname{rk}A(X)$.
\end{remark}


\begin{lemma}\label{lem:k3-hasse}
If $S$ is a projective K3 surface with $\rho(S)=1$, then
\begin{equation*}
w(-T_S)=(-1,-1).
\end{equation*}
\end{lemma}

\begin{proof}
Let $V=H^2(S,\Q)$. Recall the K3 lattice $\Lambda\simeq U^{\oplus3}\perp E_8(-1)^{\oplus2}$. 

Over $\Q$, we have $U_{\Q}\simeq\langle1,-1\rangle$. Applying formula \eqref{eq:sum} gives $\delta(U^{\oplus3})=-1$ and $w(U^{\oplus3})=(-1,-1)$. 

Let $E=E_8(-1)$. Since $\delta(E)=1$ and the Hasse invariant is
$2$-torsion, we have $w(E\perp E)=2w(E)+(1,1)=0$. Consequently, $\delta(V)=-1$ and $w(V)=(-1,-1)$.

Since $\rho(S)=1$, recall that we have $\NS(S)=\Z\ell$, and set $n=q_S(\ell,\ell)\in2\Z_{>0}$. Then $V=\langle n\rangle\perp T_S$. 
It follows that $\delta(T_S)=-n$ in $\Q^\times/\Q^{\times2}$. Hence $w(V)=w(T_S)+(n,-n)=w(T_S)$ because $(n,-n)=0$. Thus $w(T_S)=(-1,-1)$. Since $\operatorname{dim}_\Q(T_S)=m=21$, formula \eqref{eq:scale} gives $w(-T_S)=w(T_S)$. 
\end{proof}

\begin{proposition}\label{prop:hasse}
Let $X\in\mathcal C_d^{\mathrm{Hdg}}$ be rational. Then $\beta_d=0$.
\end{proposition}

\begin{proof}
Let $S$, $f$, and $c$ be as in Proposition~2.3. By Hassett~\citep[\S2]{hassett}, the integral middle cohomology is $H^4(X,\Z)\simeq\langle1\rangle^{21}\perp\langle-1\rangle^{2}$. Thus it has rank-23 and signature $(21,2)$. Consequently, $L=H^4(X,\Q)\simeq\langle1\rangle_\Q^{21}\perp\langle-1\rangle_\Q^{2}$. Therefore $\delta(L)=1$ and $w(L)=(-1,-1)$. 

The class $h^2$ in $A(X)$ has square $3$ and $A(X)_{\Z}\,\simeq K_d$ has discriminant $d$. Orthogonalizing $h^2$ in $A(X)$ therefore gives
\[
A(X)\simeq{_\Q}\langle3,\,d/3\rangle,\qquad \delta\bigl(A(X)\bigr) = d,\qquad \delta(T_X) = d .
\]
Applying \eqref{eq:sum} to $L=A(X)\perp T_X$ gives $$w(L)=(3,d/3)+w(T_X)+(d,d).$$  Since $(d,d)=(d,-1)$, it follows that
\begin{equation}\label{eq:wT}
w(T_X)=w(L)+(3,d/3)+(d,-1).
\end{equation}
Since $\operatorname{rk}T_S=\operatorname{rk}T_X=21$, we have $\rho(S)=1$, so Lemma~3.2 applies. 

By Equation~\eqref{eq:similitude}, $f$ is an isometry from $\bigl(T_X,\,-c(\,{,}\,)_X\bigr)$ onto $\bigl(T_S,(\,{,}\,)_S\bigr)$. Given Formula~\eqref{eq:scale} and the fact that the forms have rank-$21$, we have
\[
w(T_X)=w(T_S)=w(-T_S)=w(L).
\]
Comparing with \eqref{eq:wT} gives $(3,d/3)+(d,-1)=0$. The Steinberg relation $(a,1-a)=0$ gives $(2,-1)=0$ and $(3,-2)=0$, hence $(3,-1)=(3,2)=(2,3)=(2,-3)$. Therefore
\[
(3,d/3)+(d,-1)=(3,d)+(3,-1)+(d,-1)=(d,-3)+(2,-3)=(2d,-3)=\bigl(\tfrac d2,-3\bigr),
\]
using $d/3\equiv3d$ and $2d\equiv d/2$ modulo squares. Thus $\beta_d=0$.
\end{proof}

\begin{proof}[Proof of Theorem~\ref{thm:main}] The first statement is Proposition~\ref{prop:hasse}. Suppose now that $\beta_d\ne0$. The same proposition shows that no rational member of $\mathcal C_d$ lies in $\mathcal C_d^{\mathrm{Hdg}}$; hence $ \Rat\cap\mathcal C_d\subseteq \mathcal C_d\setminus\mathcal C_d^{\mathrm{Hdg}}$. Since every member of $\mathcal C_d$ has algebraic rank at least two, the complement on the right is the union in \eqref{eq:containment}. By Lemma~\ref{lem:generic}, a very general member of $\mathcal C_d$ then lies outside the rational locus. \end{proof}

 Hassett showed that $\mathcal C_{12}$ is the
closure of the locus of cubic fourfolds containing a smooth cubic scroll
\cite[Example~4.1.2]{hassett}. To our knowledge, this gives the first irrationality result for the very general member of the cubic-scroll divisor.

\begin{proof}[Proof of Corollary~\ref{cor:c12}] 
For $d=12$, we have $\beta_{12}=(6,-3)$ and $(6,-3)_2=(2,-3)_2=-1$. Thus $\beta_{12}\neq0$, and the first assertion follows from Theorem~\ref{thm:main}.
\end{proof}

\begin{corollary}
For every prime $p\equiv5\pmod6$, the very general member of
$\mathcal C_{4p}$ is irrational. In particular, this holds for
infinitely many Hassett divisors.
\end{corollary}
\begin{proof}
Let $d=4p$, where $p\equiv5\pmod6$ is prime. Then $d\equiv2\pmod6$, so $\mathcal C_d$ is nonempty, while $v_p(d/2)=1$ and $p\equiv2\pmod3$ imply $\beta_d\ne0$. Dirichlet's theorem on primes in arithmetic progressions implies that there are infinitely many primes $p\equiv5\pmod6$ and therefore infinitely many Hassett divisors $\mathcal C_{d}$ to which Theorem 1.1 applies.
\end{proof}

\begin{remark}
Fan--Lai ~\cite[Theorems~1.3 and 3.11]{Fan2020NewRC} construct rational subvarieties inside $\mathcal{C}_{20}$. However, their general members have algebraic rank-three and therefore lie in the exceptional locus \eqref{eq:containment} from~Theorem~\ref{thm:main}.
\end{remark}

\begin{proof}[Proof of Corollary \textup{\ref{cor:rational-addington}}]
Suppose that (1) holds. By
\cite[Theorem~1.3 and Equation~1.1]{Huybrechts2015TheKC}, there exists an integral Hodge isometry $T_{X,\Z}(1)\simeq T(S,\alpha, \Z)$. As recalled in Equation~\ref{eq:hub}, the twisted transcendental lattice is isometric to
$$
T(S,\alpha,\Z)
\simeq
\ker\bigl(\alpha\,\colon T_{S,\Z}\longrightarrow\Q/\Z\bigr).
$$
Since $\alpha$ is torsion, its image is finite, so $T(S,\alpha,\Z)\subset T_{S,\Z}$ has finite index. Therefore $T(S,\alpha,\Z)\otimes \Q =T_{S,\Z} \otimes \Q = T(S)_\Q$, proving (2). 

Conversely, suppose that (2) holds. After tensoring with $\Q$, the Mukai lattices decompose orthogonally as
$$
\widetilde{H}(\mathcal{A}_X, \Q) = N(\mathcal{A}_X)_{\Q} \perp T(\mathcal{A}_X)_{\Q}, \qquad 
\widetilde{H}(S, \Q) = \widetilde{\NS}(S)_{\Q} \perp T(S)_{\Q}.
$$ Because these Mukai lattices are isometric as rational quadratic spaces and, under the standard identification $T(\mathcal A_X)_\Q\simeq T_X(1)$, their transcendental parts are isometric by assumption, Witt's cancellation theorem~\cite[Theorem~2.1]{chebolu2016wittscancellationtheoremseen} then gives an isometry of their algebraic parts 
$$\psi\,\colon \widetilde{\NS}(S)_{\Q} \xrightarrow{\sim} N(\mathcal{A}_X)_{\Q}.$$
The extended rational Néron-Severi lattice decomposes orthogonally as
$$
\widetilde{\NS}(S)_{\Q} = \NS(S)_{\Q} \perp \bigl(H^0(S,\Q) \oplus H^4(S,\Q)\bigr).
$$
Choose $e=(1,0,0)$ and $f=(0,0,-1)$; under the Mukai pairing $\langle e,e\rangle=\langle f,f\rangle=0$ and $\langle e,f\rangle=1$. Hence $H^0(S,\Q) \oplus H^4(S,\Q)=\Q e \, \oplus \Q f \simeq U_\Q$. Since $H^0(S,\Q)$ and $H^4(S,\Q)$ are of Mukai type $(1,1)$,  this gives $U_\Q\hookrightarrow N(\mathcal A_X)_\Q$. 

Put $u=\psi(e)$ and $v=\psi(f)$, and choose $m\in\Z_{>0}$ such that $mu,mv\in N(\mathcal A_X)$. These vectors then define a (not necessarily primitive) embedding of the twisted hyperbolic plane
$$U(m^2) \hookrightarrow N(\mathcal{A}_X).$$
By \cite[Lemma~2.6(ii) and Proposition~2.17]{Huybrechts2015TheKC}, the existence of such
an embedding implies that there exist a
projective K3 surface $S'$ and a Brauer class
$\alpha'\in\operatorname{Br}(S')$ together with an integral Hodge
isometry
$$
\widetilde H(\mathcal A_X,\mathbb Z)
 \simeq
\widetilde H(S',\alpha',\mathbb Z).
$$
Huybrechts \cite[Theorem~1.3]{Huybrechts2015TheKC} then implies that
$X\in\mathcal C_{d'}$ for some discriminant $d\,'$ satisfying $(**')$. Because
$\operatorname{rk}A(X)_{\mathbb Z}=2$, every primitive rank-two
algebraic lattice containing $h^2$ must equal $A(X)_{\mathbb Z}$.
Since $X\in\mathcal C_d\cap\mathcal C_{d'}$, it follows that
$d\,'=d$. Consequently, $d$ satisfies $(**')$, or equivalently
$\beta_d=0$.

The equivalence of (1) and (3) follows from
\cite[Proposition~4.1]{Huybrechts2015TheKC}, together with the
rank-two uniqueness argument just given. The final assertion follows from Theorem~\ref{thm:main}.
\end{proof}

\section{The Veronese divisor and K\"uchle fourfolds of type $(c7)$}\label{sec:veronese}
Let $V=\nu_2(\PP^2)\subset\PP^5$ be a Veronese surface, and let
$X\subset\PP^5$ be a smooth cubic fourfold containing $V$.
Hassett shows that the saturated lattice
\[
K_{20}:=\langle h^2,[V]\rangle\subset A(X)_{\Z},
\]
has discriminant $d=20$, and hence $[X]\in\mathcal C_{20}$; see ~\cite[Example~4.1.4]{hassett}. 

Fixing $V$, let
\[
U_{20}\subset|\mathcal I_V(3)|
\]
be the open subset parametrizing smooth cubic fourfolds containing $V$.
Fan--Lai prove that the natural morphism
\[
\varphi\colon
[U_{20}/\operatorname{PGL}_3(\CC)]
\longrightarrow\mathcal C_{20}
\]
is birational and that its image contains
$\mathcal C_{20}\setminus\mathcal C_8$
\cite[Proposition~2.1]{Fan2020NewRC}. In particular, a very general
member of $\mathcal C_{20}$ contains a Veronese surface.

Kuznetsov proves that a smooth K\"uchle fourfold $Y$ of type
$(\mathrm{c7})$ is isomorphic to $\Bl_V(X)$
\cite[Corollary~4.11]{Kuznetsov2015OnKV}. This then identifies
an open family of K\"uchle fourfolds of type $(\mathrm{c7})$ with an
open subset of $U_{20}$. Since $\varphi$ is birational, a very general
$Y$ is therefore birational to a very general
$X\in\mathcal C_{20}$.


\begin{proof}[Proof of Corollary \textup{\ref{cor:c20}}]
For $X\in\mathcal C_{20}$, we have $\beta_{20}=(10,-3)$, and $(10,-3)_5=-1$. Thus $\beta_{20}\neq0$, and Theorem \ref{thm:main} applies. Consequently, the very general K\"uchle fourfold of type $(\mathrm{c7})$ is irrational.\end{proof}

The K\"uchle fourfold of type $(\mathrm{c7})$ is one of the
three families in K\"uchle's classification, alongside the fourfolds of type $(\mathrm{c5})$ and $(\mathrm{d3})$, whose middle cohomology is of
K3 type \cite[Theorem~3.1]{kuchle}. Benedetti--Fay--Gu\'er\'e--Manivel--Perrin recently proved that the very general fourfold of type $(\mathrm{c5})$ is irrational
\cite[Theorem~4.1]{c5}. Their criterion does not apply here as it requires $b_4(Y)_{\mathrm{van}}\geq22$, whereas
the vanishing cohomology of a fourfold of type $(\mathrm{c7})$ has rank-$21$.

\section{Noether--Lefschetz divisors of Gushel--Mukai fourfolds} 
An ordinary Gushel--Mukai fourfold is a smooth fourfold of the form
\[
X=\Gr(2,5)\cap\PP^8\cap Q\subset\PP^9,
\]
where $\Gr(2,5)\subset\PP^9$ is the Pl\"ucker embedding and
$Q\subset\PP^8$ is a quadric hypersurface. Let
\[
A(X)_\Z:=H^4(X,\Z)\cap H^{2,2}(X),
\qquad
A(X) := A(X)_\ZZ \otimes_{\Z} \Q, \qquad
T_X:=A(X)^{\perp}\subset H^4(X,\Q).
\]
The restricted Schubert classes span the primitive ambient lattice
\[
B_X:=\Z\sigma_{1,1}\oplus\Z\sigma_2\subset A(X)_\Z.
\]
In this basis its Gram matrix is
\[
\begin{pmatrix}
2&2\\
2&4
\end{pmatrix}.
\]
Following Debarre--Iliev--Manivel~\cite{debarre2015special}, we call $X$ \emph{Hodge--special} if
$\operatorname{rk}A(X)\geq 3$, or equivalently, if $A(X)$ contains a
class not coming from $\operatorname{Gr}(2,5)$. For a primitive positive-definite rank-$3$ lattice
$K\subset A(X)_\Z$ containing $B_X$, let $d$ be the discriminant of $K$.

By \cite[Corollary~6.3]{debarre2015special}, there is one
irreducible period divisor when $d\equiv0\pmod4$, while there are two when $d\equiv2\pmod8$. 
Let $\mathcal D_d$ be an irreducible component of the
discriminant--$d$ locus, and define
\[
\mathcal D_d^{\mathrm{Hdg}}
:=
\left\{
X\in\mathcal D_d:
\operatorname{rk}A(X)_{\Z}=3,\quad
\operatorname{End}_{\mathrm{Hdg}}(T_X)
=\Q\cdot\operatorname{id}_{T_X}
\right\}.
\]
Set
\[
\gamma_d:=(d,-1)\in\operatorname{Br}(\Q)[2].
\]
Note that, for $d>0$, the equality $\gamma_d=0$ is equivalent to $d=a^2+b^2$ for some $a$,~$    b\in\Z$.  

\begin{theorem}\label{thm:gm}
If $X\in\mathcal D_d^{\mathrm{Hdg}}$ is rational, then $\gamma_d=0$.
\end{theorem}

\begin{proof}
Set
\[
L_X:=H^4(X,\Q),
\qquad
B_{X,\Q}:=B_X\otimes_{\Z}\Q,
\qquad
P_X:=B_{X,\Q}^{\perp}\subset L_X.
\]
By \cite[\S5.1]{debarre2015special},
\[
L_X\simeq
\langle1\rangle^{\oplus22}\perp\langle-1\rangle^{\oplus2},
\qquad
B_{X,\Q}\simeq\langle2,2\rangle.
\]
Consequently,
\[
\delta(L_X)=\delta(B_{X,\Q})=1,
\qquad
w(L_X)=(-1,-1),
\qquad
w(B_{X,\Q})=0.
\]
Since $L_X=B_{X,\Q}\perp P_X$, 
\eqref{eq:sum} gives
\[
\delta(P_X)=1,
\qquad
w(P_X)=(-1,-1).
\]

Write $N_X:=A(X)\cap P_X$ and $A(X)=B_{X,\Q}\perp N_X$. Then
\[
P_X=N_X\perp T_X\simeq\langle d\rangle\perp T_X,
\qquad
\delta(T_X)=d.
\]
Applying \eqref{eq:sum} again  gives
\begin{equation}\label{eq:gm-hasse}
w(T_X)=(-1,-1)+(d,-1).
\end{equation}

Suppose that $X$ is rational. By Benedetti--Manivel--Perrin
\cite[Theorem~13]{benedetti2026quantumcohomologyirrationalitygushelmukai}, 
there exists a projective K3 surface $S$ and a Hodge isomorphism
$\varphi\colon P_X(1)\xrightarrow{\sim}H^2(S,\Q).$

Since $X\in\mathcal D_d^{\mathrm{Hdg}}$, the argument of
Proposition~\ref{prop:source} gives $\rho(S)=1$ and a similitude
$T_X(1)\simeq T_S$. Hence $w(T_X)=(-1,-1)$. Comparison with \eqref{eq:gm-hasse} gives
$(d,-1)=0$.
\end{proof}

The condition $\gamma_d=0$ is equivalent to every prime
$p\equiv3\pmod4$ occurring in $d$ with even exponent. This is
Pertusi's condition $(**')$, which is equivalent to $X$ having a
cohomologically associated twisted K3 surface
\cite[Theorem~1.1]{pertusi2019double}. Our obstruction is therefore
compatible with the stronger Kuznetsov--Perry conjecture that
rationality implies $\mathcal A_X\simeq D^b(S)$ for an untwisted
K3 surface $S$ \cite[Conjecture~3.12]{Kuznetsov2016DerivedCO}, where $\mathcal A_X$ is the Kuznetsov component of $X$.

\begin{corollary}\label{cor:inf}
If $(d,-1)\neq0$, then a very general member of $\mathcal D_d$ is irrational. In particular, this holds
for $d=12n^2$ with $n\geq1$, giving infinitely many Hodge--special
divisors such that the very general member is irrational.
\end{corollary}

\begin{proof}
The local period map for Gushel--Mukai fourfolds is a submersion
\cite[Theorem~4.4]{debarre2015special}. Therefore, the argument of
Lemma~\ref{lem:generic}, after replacing $\mathcal C_d$ by either irreducible
component of $\mathcal D_d$, shows that its very general member lies in
$\mathcal D_d^{\mathrm{Hdg}}$. The first assertion then follows from
Theorem~\ref{thm:gm}. For $d=12n^2$, square invariance of the Hilbert symbol gives $(d,-1)_3=(3,-1)_3=-1$,
so $(d,-1)\neq0$. Finally, every discriminant $12n^2$ occurs by
\cite[Theorem~8.1]{debarre2015special}.
\end{proof}

\begin{proof}[Proof of Corollary~\ref{cor:scroll}]
The closure of the locus of Gushel--Mukai fourfolds containing
a cubic scroll is an irreducible component of
$\mathcal D_{12}$ by \cite[Proposition~7.6]{debarre2015special}. Since $v_3(12)=1$, we have $(12,-1)\neq0$. The result therefore follows from Corollary~\ref{cor:inf}.
\end{proof}

\section{Verra fourfolds}
A Verra fourfold $V$ is a smooth double cover $\pi: V \rightarrow \PP^2 \times \PP^2$ branched over a smooth divisor of bidegree $(2,2)$. The quadric surface fibration $V \rightarrow \PP^2$, induced by composing $\pi$ with either projection $p_i: \PP^2 \times \PP^2 \rightarrow \PP^2$, has a discriminant locus that is a plane sextic curve. The double cover of $\PP^2$ branched over this sextic (when smooth) is a polarized K3 surface of degree $2$. Although a very general symmetric Verra fourfold is known to be $\ZZ/2$--irrational~\cite[Theorem~6.4]{fay2026equivariantirrationalitygeneralsymmetric}, ordinary irrationality of a very general Verra fourfold remains conjectural.

Benedetti--Guéré-Manivel-Perrin have however made recent progress on its birational geometry. In~\cite[Corollary~19]{benedetti2026quantumcohomologybirationalgeometry} they prove that a Verra fourfold is neither birational to a very general cubic nor to a very general Gushel--Mukai fourfold.  Here we extend this result to certain special cubic and Gushel--Mukai fourfolds. 


\begin{corollary}\label{cor:verra} \ 
\begin{enumerate}
    \item Let $X$ be a cubic fourfold such that $X\in\mathcal C_d^{\mathrm{Hdg}}$. If $X$ is birational to a smooth Verra fourfold then $\beta_d = 0$.
    \item Let $X$ be a Gushel--Mukai fourfold such that $X\in\mathcal D_d^{\mathrm{Hdg}}$. If $X$ is birational to a smooth Verra fourfold then $\gamma_d = 0$.
\end{enumerate}
\end{corollary}
\begin{proof}
Let $P_X = H^4(X,\Q)_{\text{prim}}$ in the cubic case and $P_X = B_{X,\Q}^\perp$ in the Gushel--Mukai case. Assume that $X$ is birational to a smooth Verra fourfold. By \cite[Theorem~21]{benedetti2026quantumcohomologybirationalgeometry}, there exists a projective K3 surface $S$ and a Hodge isomorphism $\varphi\,\colon P_X\xrightarrow{\sim}H^2(S,\Q)(-1)$. The argument of Proposition~\ref{prop:source} gives $\rho(S)=1$ and a Hodge similitude $f\colon T_X(1)\xrightarrow{\sim}T_S$. Hence $w(T_X)=(-1,-1)$. Comparison with \eqref{eq:wT} in the cubic case and \eqref{eq:gm-hasse} in the Gushel--Mukai case gives the asserted
vanishing. Hence whenever $\beta_d \ne 0$ or $\gamma_d \ne 0$, $X$ is not birational to a smooth Verra fourfold.  
\end{proof}

\begin{remark}
If $X\in C_d$ and $\beta_d \ne 0$ (respectively $X\in D_d$ and $\gamma_d \ne 0$), then any cubic fourfold (respectively Gushel--Mukai fourfold) birational to
a Verra fourfold must have higher algebraic rank or non-scalar
transcendental Hodge endomorphisms.
\end{remark}

\begin{proof}[Proof of Corollary~\ref{cor:verra-examples}]
Apply Corollary~\ref{cor:verra} and our established nonvanishing calculations for
$d=12$ and $d=20$. The final assertion follows
from the birational description of type $(\mathrm{c7})$ fourfolds
in Section~\ref{sec:veronese}.
\end{proof}

\section{Acknowledgments}
I thank Nick Addington for truly invaluable discussions and feedback during the preparation of this work. I also thank Tom Coates and J\'{e}r\'{e}my Gu\'{e}r\'{e} for providing helpful feedback and guidance. This research was funded by the EPSRC grant EP/Y028872/1. 

\bibliographystyle{alpha}
\bibliography{refs}

@article{Kuznetsov2015OnKV,
	title={On {K}{\"u}chle varieties with {P}icard number greater than 1},
	author={Alexander Kuznetsov},
	journal={Izvestiya: Mathematics},
	year={2015},
	volume={79},
	pages={698 - 709},
	url={https://api.semanticscholar.org/CorpusID:119573846}
}

@article{Fan2020NewRC,
	title={New rational cubic fourfolds arising from {C}remona transformations},
	author={Yu-Wei Fan and Kuan-Wen Lai},
	journal={Algebraic Geometry},
	year={2020},
}

@book{lecturesHuybrechts,
    author = "Huybrechts, Daniel",
    title = "{Lectures on K3 Surfaces}",
    isbn = "978-1-107-15304-2, 978-1-316-79757-0",
    publisher = "Cambridge University Press",
    year = "2016"
}

@misc{chebolu2016wittscancellationtheoremseen,
      title={Witt's cancellation theorem seen as a cancellation}, 
      author={Sunil K. Chebolu and Dan McQuillan and Jan Minac},
      year={2016},
      eprint={1106.2595},
      archivePrefix={arXiv},
      primaryClass={math.NT},
      url={https://arxiv.org/abs/1106.2595}, 
}

@article{Addington_2016,
   title={On two rationality conjectures for cubic fourfolds},
   volume={23},
   ISSN={1945-001X},
   url={http://dx.doi.org/10.4310/MRL.2016.v23.n1.a1},
   DOI={10.4310/mrl.2016.v23.n1.a1},
   number={1},
   journal={Mathematical Research Letters},
   publisher={International Press of Boston},
   author={Addington, Nicolas},
   year={2016},
   pages={1–13} }

@article{hassett,
	author = {Hassett, Brendan},
	date = {2000/01/01},
	doi = {10.1023/A:1001706324425},
	id = {Hassett2000},
	isbn = {1570-5846},
	journal = {Compositio Mathematica},
	number = {1},
	pages = {1--23},
	title = {Special Cubic Fourfolds},
	url = {https://doi.org/10.1023/A:1001706324425},
	volume = {120},
	year = {2000}}

@inbook{SE,
	title        = {A course in arithmetic},
	author       = {Serre, Jean-Pierre},
	year         = {1973},
	publisher    = {Springer-Verlag},
	isbn         = {0387900403},
	language     = {English},
	url          = {https://nla.gov.au/nla.cat-vn2542782},
}

@misc{KKPY,
	title={Birational Invariants from {H}odge Structures and Quantum Multiplication}, 
	author={Ludmil Katzarkov and Maxim Kontsevich and Tony Pantev and Tony Yue Yu},
	year={2026},
	eprint={2508.05105},
	archivePrefix={arXiv},
	primaryClass={math.AG},
	url={https://arxiv.org/abs/2508.05105}, 
}

@misc{iritani2026notesdecompositiontheoremblowups,
	title={Notes on the decomposition theorem for blowups}, 
	author={Hiroshi Iritani},
	year={2026},
	eprint={2604.10028},
	archivePrefix={arXiv},
	primaryClass={math.AG},
	url={https://arxiv.org/abs/2604.10028}, 
}

@article{kuchle,
	author = {K\"uchle, Oliver},
	date = {1995/01/01},
	doi = {10.1007/BF02571923},
	id = {K{\"u}chle1995},
	isbn = {1432-1823},
	journal = {Mathematische Zeitschrift},
	number = {1},
	pages = {563--575},
	title = {On {F}ano 4-folds of index 1 and homogeneous vector bundles over {G}rassmannians},
	url = {https://doi.org/10.1007/BF02571923},
	volume = {218},
	year = {1995}}

@misc{c5,
	title={An atomic criterion for irrationality without quantum computations}, 
	author={Vladimiro Benedetti and Aideen Fay and Jérémy Guéré and Laurent Manivel and Nicolas Perrin},
	year={2026},
	eprint={2607.26718},
	archivePrefix={arXiv},
	primaryClass={math.AG},
	url={https://arxiv.org/abs/2607.26718}, 
}

@article{Kuznetsov2016DerivedCO, title={Derived categories of {G}ushel--{M}ukai varieties}, volume={154}, DOI={10.1112/S0010437X18007091}, number={7}, journal={Compositio Mathematica}, author={Kuznetsov, Alexander and Perry, Alexander}, year={2018}, pages={1362–1406}}

@article{pertusi2019double,
	title={On the double {EPW} sextic associated to a {G}ushel--{M}ukai fourfold},
	author={Pertusi, Laura},
	journal={Journal of the London Mathematical Society},
	volume={100},
	number={1},
	pages={83--106},
	year={2019},
	publisher={Wiley Online Library}
}

@article{debarre2015special,
	title={Special prime {F}ano fourfolds of degree 10 and index 2},
	author={Debarre, Olivier and Iliev, Atanas and Manivel, Laurent},
	journal={Recent advances in algebraic geometry},
	volume={417},
	pages={123},
	year={2015},
	publisher={Cambridge University Press Cambridge}
}

@misc{benedetti2026quantumcohomologyirrationalitygushelmukai,
	title={Quantum cohomology and irrationality of {G}ushel--{M}ukai fourfolds}, 
	author={Vladimiro Benedetti and Laurent Manivel and Nicolas Perrin},
	year={2026},
	eprint={2603.17487},
	archivePrefix={arXiv},
	primaryClass={math.AG},
	url={https://arxiv.org/abs/2603.17487}, 
}

@article{Huybrechts2015TheKC,
	title={The {K3} category of a cubic fourfold},
	author={Daniel Huybrechts},
	journal={Compositio Mathematica},
	year={2015},
	volume={153},
	pages={586 - 620},
	url={https://api.semanticscholar.org/CorpusID:117981635}
}

@article{iritani2023quantum,
	title={Quantum cohomology of blowups},
	author={Iritani, Hiroshi},
	journal={arXiv preprint arXiv:2307.13555},
	year={2023}
}

@misc{guere,
	title={On the irrationality of cubic fourfolds}, 
	author={J{\'e}r{\'e}my Gu{\'e}r{\'e}},
	year={2026},
	eprint={2603.04518},
	archivePrefix={arXiv},
	primaryClass={math.AG},
	url={https://arxiv.org/abs/2603.04518}, 
}

@article{CDK,
	author  = {Cattani, Eduardo and Deligne, Pierre and Kaplan, Aroldo},
	title   = {On the locus of {H}odge classes},
	journal = {J. Amer. Math. Soc.},
	volume  = {8},
	number  = {2},
	year    = {1995},
	pages   = {483--506},
	doi     = {10.1090/S0894-0347-1995-1273413-2}
}

@article{Zarhin1983,
	author = {Zarhin, Yu.G.},
	journal = {Journal für die reine und angewandte Mathematik},
	pages = {193-220},
	title = {Hodge groups of {K3} surfaces.},
	url = {http://eudml.org/doc/152536},
	volume = {341},
	year = {1983},
}

@misc{benedetti2026quantumcohomologybirationalgeometry,
      title={Quantum cohomology and birational geometry of {V}erra fourfolds}, 
      author={Vladimiro Benedetti and Jérémy Guéré and Laurent Manivel and Nicolas Perrin},
      year={2026},
      eprint={2605.30450},
      archivePrefix={arXiv},
      primaryClass={math.AG},
      url={https://arxiv.org/abs/2605.30450}, 
}

@misc{fay2026equivariantirrationalitygeneralsymmetric,
      title={Equivariant irrationality of very general symmetric {V}erra fourfolds}, 
      author={Aideen Fay},
      year={2026},
      eprint={2605.30439},
      archivePrefix={arXiv},
      primaryClass={math.AG},
      url={https://arxiv.org/abs/2605.30439}, 
}

\end{document}